\documentclass{amsart}

\newtheorem{theorem}[equation]{Theorem}
\newtheorem{lemma}[equation]{Lemma}
\newtheorem{corollary}[equation]{Corollary}
\newtheorem{proposition}[equation]{Proposition}

\numberwithin{equation}{section}

\usepackage{amscd,amsmath,amssymb,verbatim}

\begin{document}

\title[On the integrality of factorial ratios and mirror maps]{On the integrality of factorial ratios\\ and mirror maps}
\author{Alan Adolphson}
\address{Department of Mathematics\\
Oklahoma State University\\
Stillwater, Oklahoma 74078}
\email{adolphs@math.okstate.edu}
\author{Steven Sperber}
\address{School of Mathematics\\
University of Minnesota\\
Minneapolis, Minnesota 55455}
\email{sperber@math.umn.edu}
\date{\today}
\begin{abstract}
Landau has characterized the integrality of certain ratios of factorials.  Delaygue has characterized the integrality of the Taylor coefficients of certain mirror maps constructed from series involving those ratios.  Using the $A$-hypergeometric point of view, we express those characterizations in terms of the nonexistence of interior points in multiples of the associated lattice polytope.  
\end{abstract}
\maketitle

\section{Introduction}

Let $c_{js}, d_{ks}\in{\mathbb Z}_{\geq 0}$, $1\leq j\leq J$, $1\leq k\leq K$, $1\leq s\leq r$, and let
\begin{align}
C_j(x_1,\dots,x_r) &= \sum_{s=1}^r c_{js}x_s, \\
D_k(x_1,\dots,x_r) &= \sum_{s=1}^r d_{ks}x_s. 
\end{align}
To avoid trivial cases, we assume that no $C_j$ or $D_k$ is identically zero and that $C_j\neq D_k$ for all $j$ and $k$.  We also assume that for each $s$, some $c_{js}\neq 0$ or some $d_{ks}\neq 0$, i.~e., each variable $x_s$ appears in some $C_j$ or $D_k$ with nonzero coefficient.  We always make the hypothesis that
\begin{equation}
\sum_{j=1}^J C_j(x_1,\dots,x_r) = \sum_{k=1}^K D_k(x_1,\dots,x_r),
\end{equation}
i.~e.,
\begin{equation}
\sum_{j=1}^J c_{js} = \sum_{k=1}^K d_{ks} \quad\text{for $s=1,\dots,r$.}
\end{equation}
We consider the ratios
\begin{equation}
E(m_1,\dots,m_r):=\frac{\prod_{j=1}^J C_j(m_1,\dots,m_r)!}{\prod_{k=1}^K D_k(m_1,\dots,m_r)!}
\end{equation}
for $m_1,\dots,m_r\in{\mathbb Z}_{\geq 0}$ and the series 
\begin{equation}
F(t_1,\dots,t_r) = \sum_{m_1,\dots,m_r = 0}^\infty E(m_1,\dots,m_r)t_1^{m_1}\cdots t_r^{m_r}.
\end{equation}

The series $F(t)$ is $A$-hypergeometric (see Section 3).  It is a problem of some interest to determine when a hypergeometric series has $p$-integral coefficients for a prime number $p$.  Such results were first proved by Dwork\cite{D1,D2}, some later contributions are due to Christol\cite{C} and the authors\cite{AS}.  As Dwork\cite{D1,D3} showed, such integrality results lead to $p$-adic analytic formulas for roots of zeta and $L$-functions over finite fields.  For a more recent example of this phenomenon, see~\cite{AS3}.  Dwork\cite{D2} also showed that under a certain additional condition one can prove the $p$-integrality of the Taylor coefficients of the mirror map associated to the 
$p$-integral hypergeometric series.  

Landau\cite{L} has characterized the integrality of the ratios $E(m)$.
\begin{theorem}
Assume that (1.4) holds.  One has $E(m_1,\dots,m_r)\in{\mathbb Z}$ for all $m_1,\dots,m_r\in{\mathbb Z}_{\geq 0}$ if and only if the step function
\begin{equation}
\Phi(x_1,\dots,x_r):=  \sum_{j=1}^J \lfloor C_j(x_1,\dots,x_r) \rfloor - \sum_{k=1}^K \lfloor D_k(x_1,\dots,x_r)\rfloor
\end{equation}
is $\geq 0$ for all $x_1,\dots,x_r\in [0,1)$.
\end{theorem}

{\bf Remark.}  Landau's result is valid without assumption (1.4) provided $[0,1)$ is replaced by $[0,1]$ in the statement of Theorem~1.7.  Hypothesis (1.4) implies that $\Phi(x)$ depends only on $x\;({\rm mod}\;{\mathbb Z}^r)$:  for $n\in{\mathbb Z}^r$, $\Phi(x+n) = \Phi(x)$.  This allows us to replace $[0,1]$ by $[0,1)$ in Landau's result.

Landau's result, along with a generalization of Dwork's approach, has been applied recently by Krattenthaler and Rivoal\cite{KR,KR2}, Delaygue\cite{D}, and Delaygue, Rivoal, and Roques\cite{DRR} to prove the integrality of the Taylor coefficients of certain mirror maps.  For example, Delaygue considers the series
\begin{equation}
G_{C_j}(t_1,\dots,t_r) = 
\sum_{\substack{m_1,\dots,m_r = 0\\ C_j(m)\neq 0}}^\infty E(m_1,\dots,m_r)H_{C_j(m)}t_1^{m_1}\cdots t_r^{m_r}
\end{equation}
for $j=1,\dots,J$, and
\begin{equation}
G_{D_k}(t_1,\dots,t_r) = 
\sum_{\substack{m_1,\dots,m_r = 0\\ D_k(m)\neq 0}}^\infty E(m_1,\dots,m_r)H_{D_k(m)}t_1^{m_1}\cdots t_r^{m_r}
\end{equation}
for $k=1,\dots,K$, where $m=(m_1,\dots,m_r)$ and, for a positive integer $M$, $H_M$ is the $M$-th harmonic number: $H_M = \sum_{i=1}^M \frac{1}{i}$.  
Define a subset ${\mathcal D}\subseteq[0,1)^r$ by the condition
\[ {\mathcal D} = \{x\in[0,1)^r\mid \text{$C_j(x)\geq 1$ for some $j$ or $D_k(x)\geq 1$ for some $k$}\}. \]
Note that $\Phi$ vanishes on $[0,1)^r\setminus{\mathcal D}$.  
Delaygue has characterized the integrality of the series $\exp\big(G_{C_j}(t)/F(t)\big)$ and $\exp\big(G_{D_k}(t)/F(t)\big)$.
\begin{theorem}
Assume that (1.4) holds and that the series $F(t)$ has integral coefficients (or equivalently, by Theorem 1.7, that $\Phi(x)\geq 0$ for all $x\in[0,1)^r$).  The series $\exp\big(G_{C_j}(t)/F(t)\big)$ and $\exp\big(G_{D_k}(t)/F(t)\big)$ have integral coefficients for $j=1,\dots,J$ and $k=1,\dots,K$ if and only if $\Phi(x)\geq 1$ for all $x\in{\mathcal D}$.
\end{theorem}

{\bf Remark.}  Delaygue's results\cite[Theorems 1, 2, and~3]{D} are more general than Theorem~1.11.  We have abridged them here because of our interest in the series (1.9) and (1.10), which we discuss in Section 4.

We describe a polytope that encapsulates the conditions of Theorems~1.7 and~1.11.  
Put $n= r+J+K$.  Let ${\bf a}_1,\dots,{\bf a}_n$ be the standard unit basis vectors in ${\mathbb R}^n$ and for $s=1,\dots,r$ let
\[ {\bf a}_{n+s} = (0,\dots,0,1,0,\dots,0,c_{1s},\dots,c_{Js},-d_{1s},\dots,-d_{Ks}), \]
where the first $r$ coordinates have a $1$ in the $s$-th position and zeros elsewhere.  Our hypothesis that some $c_{js}$ or some $d_{ks}$ is nonzero implies that ${\bf a}_1,\dots,{\bf a}_{n+r}$ are all distinct.  Put $N=n+r$ and let $A = \{{\bf a}_i\}_{i=1}^{N}\subseteq {\mathbb Z}^n$.   Let $\Delta(A)$ be the convex hull of $A\cup\{{\bf 0}\}$.  

Under Hypothesis (1.4) the elements of the set $A$ all lie on the hyperplane $\sum_{i=1}^n u_i =1$ in ${\mathbb R}^n$,
which implies that the corresponding $A$-hypergeometric system is regular holonomic (Hotta\cite[Section 6]{H}, see also Saito-Sturmfels-Takayama\cite[Theorem 2.4.9]{SST}). 

For $z\in{\mathbb R}_{\geq 0}$, let
\[ z\Delta(A) = \{ (zu_1,\dots,zu_n)\mid (u_1,\dots,u_n)\in\Delta(A)\}. \]
The characterizations in Theorems 1.7 and 1.11 can be expressed in terms of the polytope $\Delta(A)$.  
\begin{theorem}
Assume that (1.4) holds. \\
{\bf (a)}  One has $\Phi(x)\geq 0$ for all $x\in[0,1)^r$ if and only if $(J+r)\Delta(A)$ contains no interior lattice points. \\
{\bf (b)}  One has $\Phi(x)\geq 1$ for all $x\in{\mathcal D}$ if and only if $\sum_{i=1}^{J+r} {\bf a}_i$ is the unique interior lattice point of $(J+r+1)\Delta(A)$.
\end{theorem}

{\bf Remark.}  The point $\sum_{i=1}^{J+r}{\bf a}_i$ is never an interior lattice point of $(J+r)\Delta(A)$ (it lies on the face contained in the hyperplane $\sum_{i=1}^n u_i=J+r$) but is always an interior lattice point of the polytope $(J+r+1)\Delta(A)$ (see the remark following Lemma~2.5).  Thus $(J+r)\Delta(A)$ is the largest integral multiple of $\Delta(A)$ whose interior can be lattice-point free.  It will follow from part (a) and Theorem~1.7 that if the series~(1.6) has integral coefficients, then $K>J$ (see Corollary~2.22).

Our choice of the set $A$ was motivated by considering the series (1.6).
In \cite[Section 6]{A} the first author gave an algorithm, based on work of Dwork-Loeser\cite{DL}, that associates to a series whose coefficients are products of Pochhammer symbols a set $A$ with the property that the corresponding $A$-hypergeometric system has that series as a solution.  The set $A$ defined above was obtained by applying that algorithm to the series (1.6).  That series arises as a specialization of a solution to the $A$-hypergeometric system with parameter $-\sum_{i=1}^{J+r}{\bf a}_i$.  We discuss the connection with $A$-hyperg\-eometric series in Section 3.

It is well known that the polytope $\Delta(A)$ contains cohomological information about the toric hypersurfaces whose Picard-Fuchs equations are $A$-hypergeometric.  In particular, interior lattice points of multiples of $\Delta(A)$ play a role in describing the mixed Hodge structure on de Rham cohomology (Batyrev\cite{Ba}).  We plan to study the cohomological consequences of the lattice point conditions of Theorem 1.12 in a future article.  

For complete intersections in the torus, we conjectured that uniqueness of a certain interior lattice point implies  integrality of the Taylor coefficients of related mirror maps (\cite[Conjecture~6.3]{AS2}).  We regard the conjunction of Theorem~1.12(b) and Delaygue's results as evidence supporting that conjecture.

\section{Proof of Theorem 1.12}

We begin by expressing the lattice point conditions of Theorem~1.12 in terms of the cone $C(A)\subseteq{\mathbb R}^n$ generated by $A$.  Note that a point $u=(u_1,\dots,u_n)\in C(A)$ lies in $z\Delta(A)$ (where $z\geq 0$) if and only if $\sum_{i=1}^n u_i\leq z$.  The point $u$ is an interior point of $z\Delta(A)$ if and only if it is an interior point of $C(A)$ and $\sum_{i=1}^n u_i<z$.  Theorem~1.12 can therefore be reformulated as follows.
\begin{theorem}
Assume that (1.4) holds.\\
{\bf (a)}  Every interior lattice point $(u_1,\dots,u_n)$ of $C(A)$ satisfies the inequality
\begin{equation}
\sum_{i=1}^n u_i\geq J+r
\end{equation}
if and only if $\Phi(x)\geq 0$ for all $x\in[0,1)^r$.\\
{\bf (b)}  The point $\sum_{i=1}^{J+r} {\bf a}_i = (1,\dots,1,0,\dots,0)$ ($J+r$ ones, followed by $K$ zeros) is the unique interior lattice point of $C(A)$ that satisfies the inequality
\begin{equation}
\sum_{i=1}^n u_i< J+r+1
\end{equation}
if and only if $\Phi(x)\geq 1$ for all $x\in{\mathcal D}$.
\end{theorem}

We list some properties of the cone $C(A)$.  The first is very general.
\begin{lemma}
Let $B=\{{\bf b}_1,\dots,{\bf b}_M\}\subseteq{\mathbb R}^n$, let $C(B)\subseteq{\mathbb R}^n$ be the cone generated by~$B$, and let ${\mathcal M}$ be a subset of $\{1,\dots,M\}$.  If $u$ is an interior point of $C(B)$, then there exists a representation
\[ u= \sum_{i=1}^M z_i{\bf b}_i \]
with $z_i\geq 0$ for all $i$ and $z_i>0$ for $i\in {\mathcal M}$.
\end{lemma}

\begin{proof}
Since $u$ is interior to $C(B)$, we can choose $\epsilon>0$ so that $u-\epsilon\sum_{i\in {\mathcal M}} {\bf b}_i$ lies in~$C(B)$.  There then exist $y_i\geq 0$ such that
\[ u-\epsilon\sum_{i\in {\mathcal M}} {\bf b}_i = \sum_{i=1}^M y_i{\bf b}_i. \]
Solving for $u$ gives the desired representation. 
\end{proof}

\begin{lemma}
Let $u\in C(A)$ and write 
\[ u= \sum_{i=1}^{N} z_i {\bf a}_i \]
with $z_i\geq 0$ for $i=1,\dots,N$.  If $z_i>0$ for $i=1,\dots,r+J$, then $u$ is an interior point of $C(A)$.
\end{lemma}

\begin{proof}
We have the relation
\begin{equation}
\sum_{s=1}^r {\bf a}_{s} + \sum_{j=1}^J \bigg(\sum_{s=1}^r c_{js}\bigg){\bf a}_{r+j}  = \sum_{k=1}^K \bigg(\sum_{s=1}^r d_{ks}\bigg) {\bf a}_{r+J+k} + \sum_{s=1}^r {\bf a}_{n+s}. 
\end{equation}
The coefficient of each ${\bf a}_i$ is $>0$ since we assumed that no $C_j$ or $D_k$ is identically zero.  Let $v\in C(A)$ be the point represented by either side of~(2.6).  
Since every element of $A$ appears with $>0$ coefficient in~(2.6), the point $v$ cannot lie on any codimension-one face of $C(A)$.  In particular, the points $\{{\bf a}_i\}_{i=1}^{r+J}$ do not all lie on any codimension-one face of~$C(A)$.  This fact implies the lemma.
\end{proof}

\noindent{\bf Remark.}  By Lemma 2.5, the point $\sum_{i=1}^{r+J} {\bf a}_i$ is an interior lattice point of~$C(A)$.  The sum of its coordinates is $J+r$, hence it is an interior lattice point of the polytope $(J+r+1)\Delta(A)$.  Thus $(J+r)\Delta(A)$ is the largest integral multiple of $\Delta(A)$ whose interior can be lattice-point free.

\begin{proof}[Proof of Theorem 2.1]
Let $u=(u_1,\dots,u_n)$ be an interior lattice point of $C(A)$ for which $\sum_{i=1}^n u_i$ is minimal and write
\begin{equation}
(u_1,\dots,u_n) = \sum_{i=1}^N z_i{\bf a}_i
\end{equation}
with $z_i\geq 0$ for $i=1,\dots,N$.  By Lemma 2.4 we may assume that $z_i>0$ for $i=1,\dots,r+J$.  Note that since the coordinates of each ${\bf a}_i$ sum to~1 we have
\begin{equation}
\sum_{i=1}^n u_i = \sum_{i=1}^N z_i.
\end{equation}

We first claim that $z_i\leq 1$ for all $i$.  If some $z_{i_0}>1$, then
\begin{equation}
 u-{\bf a}_{i_0} = (z_{i_0}-1){\bf a}_ {i_0} + \sum_{\substack{i=1\\ i\neq i_0}}^N z_i{\bf a}_i 
\end{equation}
is an interior lattice point of $C(A)$ since every ${\bf a}_i$ with $>0$ coefficient in (2.7) occurs with $>0$ coefficient in (2.9).  But by (2.8) the sum of the coordinates of $u-{\bf a}_{i_0}$ is strictly smaller than the sum of the coordinates of $u$, contradicting the minimality property of $u$.

We claim that $z_i<1$ for $i=r+J+1,\dots,N$.  If $z_{i_0}=1$ for some $i_0\in\{r+J+1,\dots,N\}$, then (2.9) becomes
\[ u-{\bf a}_{i_0} = \sum_{\substack{i=1\\ i\neq i_0}}^N z_i{\bf a}_i. \]
But since $z_i>0$ for $i=1,\dots,r+J$, the point $u-{\bf a}_{i_0}$ is an interior lattice point of $C(A)$ by Lemma 2.5, again contradicting the minimality property of $u$.  

We have proved that in the representation (2.7) one has 
\begin{equation}
z_i\in(0,1]\quad\text{for $i=1,\dots,r+J$} 
\end{equation}
and
\begin{equation}
z_i\in [0,1)\quad\text{for $i=r+J+1,\dots,N$.} 
\end{equation}

We now examine (2.7) coordinatewise.  For $i=1,\dots,r$ we have
\begin{equation}
 u_i = z_i + z_{n+i}. 
\end{equation}
By (2.10) and (2.11) we have $z_i\in(0,1]$ and $z_{n+i}\in[0,1)$.  Since $u_i\in{\mathbb Z}$, Equation~(2.12) implies
\begin{equation}
u_i=1 \quad\text{for $i=1,\dots,r$}
\end{equation}
and
\begin{equation}
z_i=1-z_{n+i} \quad\text{for $i=1,\dots,r$.}
\end{equation}
For $j=1,\dots,J$ we have
\begin{equation}
 u_{r+j} = z_{r+j}+C_j(z_{n+1},\dots,z_{n+r}). 
\end{equation}
Since $u_{r+j}\in{\mathbb Z}$ and $z_{r+j}\in(0,1]$ we have
\begin{equation}
z_{r+j} = 1 + \big\lfloor C_j(z_{n+1},\dots,z_{n+r})\big\rfloor - C_j(z_{n+1},\dots,z_{n+r}) \quad\text{for $j=1,\dots,J$,}
\end{equation}
which implies by (2.15)
\begin{equation}
u_{r+j}= 1 + \big\lfloor C_j(z_{n+1},\dots,z_{n+r})\big\rfloor\quad\text{for $j=1,\dots, J$.}
\end{equation}
For $k=1,\dots,K$ we have
\begin{equation}
 u_{r+J+k} = z_{r+J+k} -D_k(z_{n+1},\dots,z_{n+r}). 
\end{equation}
Since $u_{r+J+k}\in{\mathbb Z}$ and $z_{r+J+k}\in[0,1)$ we have
\begin{equation}
z_{r+J+k} = D_k(z_{n+1},\dots,z_{n+r}) - \big\lfloor D_k(z_{n+1},\dots,z_{n+r})\big\rfloor\quad\text{for $k=1,\dots,K,$}
\end{equation}
which implies by (2.18)
\begin{equation}
u_{r+J+k} = - \big\lfloor D_k(z_{n+1},\dots,z_{n+r})\big\rfloor\quad\text{for $k=1,\dots,K$.}
\end{equation}

Adding (2.13), (2.17), and (2.20) gives
\begin{align}
\sum_{i=1}^n u_i &= r+J+\sum_{j=1}^J \big\lfloor C_j(z_{n+1},\dots,z_{n+r}) \big\rfloor - \sum_{k=1}^K
\big\lfloor D_k(z_{n+1},\dots,z_{n+r})\big\rfloor \\ \nonumber
 &= J+r+\Phi(z_{n+1},\dots,z_{n+r}).
\end{align}
We conclude that if $\Phi(x)\geq 0$ for all $x\in[0,1)^r$, then every interior lattice point $(u_1,\dots,u_n)$ of $C(A)$ satisfies (2.2).  This proves one direction of (a).

If $\Phi(x)\geq 1$ for all $x\in{\mathcal D}$, then the right-hand side of (2.21) takes its minimal value of $J+r$ for $(z_{n+1},\dots,z_{n+r})\in [0,1)^r\setminus {\mathcal D}$, where $\Phi$ vanishes.  But by the definition of ${\mathcal D}$, if
$(z_{n+1},\dots,z_{n+r})\in [0,1)^r\setminus {\mathcal D}$, then 
\[ C_j(z_{n+1},\dots,z_{n+r}) < 1\quad\text{for all $j$} \]
and
\[ D_k(z_{n+1},\dots,z_{n+r}) < 1\quad\text{for all $k$.} \]
Equations (2.13) and (2.17) then imply that $u_i=1$ for $i=1,\dots,r+J$ and (2.20) implies that $u_i=0$ for $i=r+J+1,\dots,r+J+K$, hence $(u_1,\dots,u_n) = \sum_{i=1}^{r+J}{\bf a}_i$.  This shows that $\sum_{i=1}^{r+J}{\bf a}_i$ is the unique interior lattice point of $C(A)$ satisfying $\sum_{i=1}^n u_i<J+r+1$.  This proves one direction of (b).

To prove the other directions of (a) and (b), let $z_{n+1},\dots z_{n+r}\in [0,1)$ and define $z_i$ for $i=1,\dots,n$ by (2.14), (2.16), and~(2.19).  This gives a sequence $z_1,\dots,z_N$ satisfying (2.10) and~(2.11).  Now define $u\in C(A)$ by (2.7).  By (2.10) and Lemma~2.5, the point $u$ is an interior point of $C(A)$.  Equation (2.7) implies that (2.12), (2.15), and (2.18) hold.  Our definitions of $z_1,\dots,z_N$ then imply that (2.13), (2.17), and (2.20) hold.  These show that $(u_1,\dots,u_n)$ is a lattice point and that (2.21) holds.  

If $\Phi(z_{n+1},\dots,z_{n+r})<0$, Eq.~(2.21) implies that (2.2) fails for $u$, proving the other direction of (a).  If $(z_{n+1},\dots,z_{n+r})\in{\mathcal D}$ and $\Phi(z_{n+1},\dots,z_{n+r})<1$, then Eq.~(2.21) implies that $u$ satisfies (2.3).  Furthermore, since $(z_{n+1},\dots,z_{n+r})\in{\mathcal D}$ we have either 
\[ C_j(z_{n+1},\dots,z_{n+r})\geq 1\quad\text{for some $j$}, \]
which implies by (2.17) that $u_{r+j}>1$, or
\[ D_k(z_{n+1},\dots,z_{n+r})\geq 1\quad\text{for some $k$}, \]
which implies by (2.19) that $u_{r+J+k}<0$.  In either case we have $u\neq\sum_{i=1}^{J+r}{\bf a}_i$, hence $\sum_{i=1}^{J+r}{\bf a}_i$ is not the only interior lattice point of $C(A)$ that satisfies (2.3).  This proves the other direction of (b).
\end{proof}

\begin{corollary}
If $K\leq J$, then the series (1.6) does not have integral coefficients.
\end{corollary}

\begin{proof}
Lemma 2.5 and Equation (2.6) imply that $\sum_{k=1}^K{\bf a}_{r+J+K} + \sum_{s=1}^r{\bf a}_{n+s}$ is an interior point of $C(A)$.  The sum of its coordinates is $K+r$, so it is an interior point of $(K+r+1)\Delta(A)$.  If $K<J$, this implies that $(J+r)\Delta(A)$ contains an interior lattice point so the assertion of the corollary follows from Theorems~1.12(a) and~1.7.

Suppose that $K=J$.  If (1.6) has integral coefficients, then by Theorems 1.7 and~1.12(a) the polytope $(J+r)\Delta(A)$ has no interior lattice points.  Interchange the roles of the $C_j$ and $D_k$: let
\begin{equation}
E'(m_1,\dots,m_r)=\frac{\prod_{k=1}^K D_k(m_1,\dots,m_r)!}{\prod_{j=1}^J C_j(m_1,\dots,m_r)!}
\end{equation}
for $m_1,\dots,m_r\in{\mathbb Z}_{\geq 0}$ and let
\begin{equation}
F'(t_1,\dots,t_r) = \sum_{m_1,\dots,m_r = 0}^\infty E'(m_1,\dots,m_r)t_1^{m_1}\cdots t_r^{m_r}.
\end{equation}

Let $A'\subseteq{\mathbb R}^n$ be the set of lattice points consisting of the standard unit basis vectors $\{{\bf a}_i\}_{i=1}^n$ and the vectors
\[ {\bf a}'_{n+s} = (0,\dots,0,1,0,\dots,0,d_{1s},\dots,d_{Ks},-c_{1s},\dots,-c_{Js}) \]
for $s=1,\dots,r$.  Let $\Delta(A')$ be the convex hull of $A'\cup\{{\bf 0}\}$.  The polytope $\Delta(A')$ is the image of the polytope $\Delta(A)$ under the linear transformation on ${\mathbb R}^n$ defined by $u_s\mapsto u_s$ for $s=1,\dots,r$, $u_{r+j}\mapsto -u_{r+K+j}$ for $j=1,\dots,J$, and $u_{r+J+k}\mapsto -u_{r+k}$ for $k=1,\dots,K$.  This is a unimodular transformation, so it preserves combinatorial properties.  In particular, the polytope $(K+r)\Delta(A')$ contains no interior lattice points since $K=J$.  Theorems 1.7 and 1.12(a) then imply that the series (2.24) has integral coefficients.  But these coefficients are the reciprocals of the coefficients of~(1.6), which we are assuming has integral coefficients, so all these coefficients must equal~1.  We therefore have
\begin{equation}
\prod_{j=1}^J C_j(m_1,\dots,m_r)! = \prod_{k=1}^K D_k(m_1,\dots,m_r)!
\end{equation}
for all $m_1,\dots,m_r\in{\mathbb Z}_{\geq 0}$.  

We claim that if (2.25) holds, then $(C_j)_{j=1}^J$ is a rearrangement of $(D_k)_{k=1}^K$, contradicting our hypothesis that $C_j\neq D_k$ for all $j$ and $k$.  
We consider first the case $r=1$ and proceed by induction on $J$.  The case $J=1$ is trivial, so suppose that $J>1$ and
\begin{equation}
\prod_{j=1}^J (c_{j1}m_1)! = \prod_{k=1}^K (d_{k1}m_1)!
\end{equation}
for all nonnegative integers $m_1$.  Assume the coefficients are ordered so that
\[ c_{11} = \max_{j=1,\dots,J}\{c_{j1}\} \quad\text{and}\quad d_{11} = \max_{k=1,\dots,K}\{d_{k1}\}. \]
Suppose that $c_{11}\neq d_{11}$, say, $c_{11}>d_{11}$.  We cannot have $d_{11}=1$.  For then $d_{k1} =1$ for all $k$, but $c_{11}>1$, $c_{j1}>0$ for all $j$ (by hypothesis), and $J=K$ make it impossible to satisfy Hypothesis~(1.4).  

Choose a prime number $\pi>d_{11}^2 + d_{11}$ and write
$\pi = qd_{11}+r$ with $0<r<d_{11}$ (we have $r>0$ since $d_{11}\neq 1$ and $\pi$ is prime).  Then $q>d_{11}$ since $\pi>d_{11}^2 + d_{11}$, so
\[ qc_{11}\geq q(d_{11}+1) = qd_{11} + q >qd_{11}+d_{11}>qd_{11} + r =\pi >qd_{11}. \]
If we take $m_1=q$ in (2.26), then the prime number $\pi$ divides the left-hand side of (2.26) but not the right-hand side, a contradiction.  Therefore $c_{11} = d_{11}$, so these factors may be canceled from (2.26) and the claim follows by induction on~$J$.

Consider now the case of general $r$, where we again proceed by induction on $J$.  For $J=1$ we have
\begin{equation}
(c_{11}m_1 + \cdots + c_{1r}m_r)! = (d_{11}m_1 + \cdots + d_{1r}m_r)!
\end{equation}
for all nonegative integers $m_1,\dots,m_r$.  This implies that the linear forms $C_1$ and $D_1$ assume equal values at all points of $\big({\mathbb Z}_{\geq 0}\big)^r$, hence $C_1=D_1$.  Now suppose that $J>1$.  Fix nonnegative integers $m_1,\dots,m_r$.  Equation (2.25) with $m_s$ replaced by $m_s\cdot m$ for all $s$ gives
\begin{equation}
\prod_{j=1}^J \big(C_j(m_1,\dots,m_r)m\big)! = \prod_{k=1}^K \big(D_k(m_1,\dots,m_r)m\big)!
\end{equation}
for all nonnegative integers $m$.  The case $r=1$ now implies that the integers $C_j(m_1,\dots,m_r)$ are a rearrangement of the integers $D_k(m_1,\dots,m_r)$.  Take $j=1$ and for $k=1,\dots,K$ define
\[ {\mathcal M}_k = \{ (m_1,\dots,m_r)\in\big({\mathbb Z}_{\geq 0}\big)^r\mid C_1(m_1,\dots,m_r) = D_k(m_1,\dots,m_r)\}. \]
Since $\big({\mathbb Z}_{\geq 0}\big)^r = \bigcup_{k=1}^K {\mathcal M}_k$, at least one of the sets ${\mathcal M}_k$ is not contained in a hyperplane.  For such a $k$ we have $C_1 = D_k$, so we may cancel these factors from (2.25) and we are done by induction on $J$.
\end{proof}

{\bf Remark.}  It is straightforward to check that
\[ \max_j\{c_{js}\}\geq \max_k\{d_{ks}\}\quad\text{for $s=1,\dots,r$} \]
is a necessary condition for the integrality of all $E(m)$.

\section{$A$-hypergeometric series}

In this section we describe the $A$-hypergeometric nature of the series~(1.6).  We begin by recalling the definition of the $A$-hypergeometric system of differential equations.

Let ${L}\subseteq{\mathbb Z}^N$ be the lattice of relations on $A$:
\[ {L} = \bigg\{ l=(l_1,\dots,l_N)\in{\mathbb Z}^N\:\bigg|\: \sum_{i=1}^N l_i{\bf a}_i = {\bf 0} \bigg\}. \]
Let $\beta = (\beta_1,\dots,\beta_n)\in{\mathbb C}^n$.  The {\it $A$-hypergeometric system with parameter $\beta$\/} is the system of partial differential operators in $\lambda_1,\dots,\lambda_N$ consisting of the {\it box operators\/}
\begin{equation}
\Box_l = \prod_{l_i>0}\bigg(\frac{\partial}{\partial\lambda_i}\bigg)^{l_i} - \prod_{l_i<0}\bigg(\frac{\partial}{\partial \lambda_i}\bigg)^{-l_i}\quad\text{for $l\in {L}$}
\end{equation}
and the {\it Euler\/} or {\it homogeneity operators}
\begin{equation}
Z_i = \sum_{j=1}^N a_{ij}\lambda_j\frac{\partial}{\partial \lambda_j} -\beta_i \quad\text{for $i=1,\dots,n$,}
\end{equation}
where ${\bf a}_j = (a_{1j},\dots,a_{nj})$.  

To describe solutions of these systems, it will be convenient to replace the classical Pochhammer symbol with a different notation.  Define for $z\in{\mathbb C}$ and $k\in {\mathbb Z}$, $k<-z$ if $z\in{\mathbb Z}_{<0}$,
\[ [z]_k = \begin{cases} 1 & \text{if $k=0$,} \\ \displaystyle\frac{1}{(z+1)(z+2)\cdots(z+k)} & \text{if $k>0$,} \\
z(z-1)\cdots(z+k+1) & \text{if $k<0$.} \end{cases} \]

For $z=(z_1,\dots,z_N)\in{\mathbb C}^N$ and $k=(k_1,\dots,k_N)\in{\mathbb Z}^N$ we define
\[ [z]_k = \prod_{i=1}^N [z_i]_{k_i}. \]
The {\it negative support\/} of $z$ is the set
\[ {\rm nsupp}(z) = \{i\in \{1,\dots,N\}\mid \text{$z_i$ is a negative integer}\}. \]

Let $v=(v_1,\dots,v_N)\in{\mathbb C}^N$ satisfy $\sum_{i=1}^N v_i{\bf a}_i = \beta$.  One says that $v$ has {\it minimal negative support\/} if there is no $l\in {L}$ for which ${\rm nsupp}(v+l)$ is a proper subset of ${\rm nsupp}(v)$.  Let
\[ {L}_v = \{l\in {L}\mid {\rm nsupp}(v+l) = {\rm nsupp}(v)\} \]
and let
\begin{equation}
F_v(\lambda) = \sum_{l\in {L}_v} [v]_l\lambda^{v+l}.
\end{equation}
By \cite[Proposition 3.4.13]{SST}, the series $F_v(\lambda)$ is a solution of the $A$-hypergeometric system (3.1), (3.2) if and only if $v$ has minimal negative support.

For the set $A$ defined in Section 1, we give an explicit description of ${L}$.  Let $l=(l_1,\dots,l_N)\in {L}$.  Recall that $n=r+J+K$ and $N=n+r$.  We have from the definition of the~${\bf a}_i$
\begin{multline}
{\bf 0} = \sum_{i=1}^N l_i{\bf a}_i = \big(l_1+l_{n+1},\dots,l_r+l_{n+r},\\
l_{r+1}+C_1(l_{n+1},\dots,l_{n+r}),\dots,l_{r+J}+C_J(l_{n+1},\dots,l_{n+r}),\\
l_{r+J+1}-D_1(l_{n+1},\dots,l_{n+r}),\dots,l_{r+J+K}-D_K(l_{n+1},\dots,l_{n+r})\big).
\end{multline}
For notational convenience, set $p_s=l_{n+s}$ for $s=1,\dots,r$ and put $p=(p_1,\dots,p_r)$.  Since every coordinate on the right-hand side of (3.4) must vanish, we get
\begin{multline}
{L}=\{ l=(-p_1,\dots,-p_r,-C_1(p),\dots,-C_J(p),\\ 
D_1(p),\dots,D_K(p),p_1,\dots,p_r)\mid p=(p_1,\dots,p_r)\in{\mathbb Z}^r\}.
\end{multline}

We now explain how to choose $v$ and $\beta$ so that the series (3.3) becomes the series~(1.6).  Put
\[ v^{(0)} = (-1,\dots,-1,0,\dots,0) \]
($-1$ repeated $r+J$ times, $0$ repeated $K+r$ times).  This gives
\[ \beta^{(0)} = \sum_{i=1}^{r+J} (-1){\bf a}_i = (-1,\dots,-1,0,\dots,0) \]
($-1$ repeated $r+J$ times, $0$ repeated $K$ times).  For $l\in {L}$ as given in~(3.5), we see that ${\rm nsupp}(v^{(0)}) = {\rm nsupp}(v^{(0)}+l)$ if and only if 
\[ p_1,\dots,p_r,C_1(p),\dots,C_J(p),D_1(p),\dots,D_K(p)\geq 0. \]
Since the $C_j$ and $D_k$ have nonnegative coefficients, this is equivalent to assuming $p_1,\dots,p_r\geq 0$.  Thus
\begin{multline*}
{L}_{v^{(0)}} = \{ l=(-p_1,\dots,-p_r,-C_1(p),\dots,-C_J(p),\\ 
D_1(p),\dots,D_K(p),p_1,\dots,p_r)\mid p=(p_1,\dots,p_r)\in\big({\mathbb Z}_{\geq 0}\big)^r\}
\end{multline*}
and the series (3.3) becomes
\begin{multline*}
F_{v^{(0)}}(\lambda) = (\lambda_1\cdots\lambda_{r+J})^{-1}\cdot \\
\sum_{p_1,\dots,p_r=0}^\infty  \bigg(\prod_{s=1}^r [-1]_{-p_s} \prod_{j=1}^J [-1]_{-C_j(p)}\prod_{k=1}^K [0]_{D_k(p)}\prod_{s=1}^r [0]_{p_s}\bigg) \lambda^l. 
\end{multline*}
Using (3.5) and the definition of the symbol $[z]_k$, this simplifies to
\begin{multline}
F_{v^{(0)}}(\lambda) = (\lambda_1\cdots\lambda_{r+J})^{-1}\cdot \\
\sum_{p_1,\dots,p_r=0}^\infty (-1)^{\sum_{s=1}^r p_s + \sum_{j=1}^J C_j(p)} \frac{\displaystyle\prod_{j=1}^J C_j(p)!}{\displaystyle\prod_{k=1}^K D_k(p)!} \frac{\displaystyle\prod_{k=1}^K \lambda_{r+J+k}^{D_k(p)}\prod_{s=1}^r \lambda_{n+s}^{p_s}}{\displaystyle\prod_{s=1}^r \lambda_s^{p_s}\prod_{j=1}^J \lambda_{r+j}^{C_j(p)}}.
\end{multline}
If we make the specializations $\lambda_i\mapsto 1$ for $i=1,\dots,r+J$, $\lambda_i\mapsto -1$ for $i=r+J+1,\dots,r+J+K$, and $\lambda_{n+s}\mapsto -t_s$ for $s=1,\dots,r$, then using (1.3) this becomes the series~(1.6).

As noted above, the series $F_{v^{(0)}}(\lambda)$ will be a solution of the $A$-hypergeometric system with parameter $\beta^{(0)}$ if and only if $v^{(0)}$ has minimal negative support.  From (3.5) we see that $v^{(0)}$ will {\it fail\/} to have minimal negative support if and only if there exists $p=(p_1,\dots,p_r)\in\big({\mathbb Z}_{\geq 0}\big)^r$ such that $D_k(p)\geq0$ for $k=1,\dots,K$ and $C_j(p)<0$ for some $j\in\{1,\dots,J\}$.  But since all $C_j$ have nonnegative coefficients, this cannot occur.
\begin{proposition}
The series $F_{v^{(0)}}(\lambda)$ given by (3.6) satisfies the $A$-hypergeometric system (3.1), (3.2) with parameter $\beta^{(0)}$.
\end{proposition}

{\bf Remark.}  It follows from \cite[Proposition 5.2]{AS} that if $(J+r)\Delta(A)$ contains no interior lattice points, then $F_{v^{(0)}}(\lambda)$ has integral coefficients.  It was this earlier result that motivated our current study.  

\section{Logarithmic solutions}

In this section we show that the series (1.9) and (1.10) are specializations of logarithmic solutions of the $A$-hyper\-geometric system with parameter $\beta^{(0)}$.  We begin by reviewing some results from \cite{AS2}.

For $i=1,\dots,N$ and $v=(v_1,\dots,v_N)\in{\mathbb C}^N$, we define the {\it $\hat{\imath}$-negative support\/} of~$v$ to be
\[ \text{$\hat{\imath}$-}{\rm nsupp}(v) = \{ j\in\{1,\dots,\hat{\imath},\dots,N\}\mid\text{$v_j$ is a negative integer,}\} \]
where the symbol `$\hat{\imath}$' indicates that the element $i$ has been omitted from the set.  We say that $v$ has {\it minimal $\hat{\imath}$-negative support\/} if $\text{$\hat{\imath}$-}{\rm nsupp}(v+l)$ is not a proper subset of $\text{$\hat{\imath}$-}{\rm nsupp}(v)$ for any $l\in {L}$.  Let
\[ {L}_{v,\hat{\imath}} = \{ l\in {L}\mid \text{$\hat{\imath}$-}{\rm nsupp}(v+l)=\text{$\hat{\imath}$-}{\rm nsupp}(v)\}. \]
Note that ${L}_v\subseteq {L}_{v,\hat{\imath}}$ for all $i$.  

Using (3.5), one checks that the vector $v^{(0)}$ of Section 3 has minimal $\hat{\imath}$-negative support for $i=1,\dots,n$ and that ${L}_{v^{(0)},\hat{i}} = {L}_{v^{(0)}}$ for these $i$.  We can then apply \cite[Theorem~4.11]{AS2} to find series $G_i(\lambda)$ for $i=1,\dots,n$ such that the expressions $F_{v^{(0)}}(\lambda)\log\lambda_i + G_i(\lambda)$ are what we call {\it quasi-solutions\/} of the $A$-hypergeometric system with parameter $\beta^{(0)}$:  $F_{v^{(0)}}(\lambda)$ is a solution of the $A$-hypergeometric system (3.1), (3.2) with parameter $\beta^{(0)}$, $G_i(\lambda)$ satisfies the Euler operators (3.2) with parameter $\beta^{(0)}$, and $F_{v^{(0)}}(\lambda)\log\lambda_i + G_i(\lambda)$ satisfies the box operators~(3.1).

The formula of \cite[Theorem~4.11]{AS2} gives explicit expressions for these $G_i$.   We record them here for $i=r+1,\dots,r+J+K$, for which the $G_i$ are related to the series (1.9) and (1.10).  For $i=r+j_0$ with $1\leq j_0\leq J$ we have 
\begin{multline}
G_{r+j_0}(\lambda) = -(\lambda_1\cdots\lambda_{r+J})^{-1} \sum_{\substack{p_1,\dots,p_r=0\\ C_{j_0}(p)\neq 0}}^\infty 
 (-1)^{\sum_{s=1}^r p_s + \sum_{j=1}^J C_j(p)}  \\ \cdot H_{C_{j_0}(p)}
\frac{\displaystyle\prod_{j=1}^J C_j(p)!}{\displaystyle\prod_{k=1}^K D_k(p)!} \frac{\displaystyle\prod_{k=1}^K \lambda_{r+J+k}^{D_k(p)}\prod_{s=1}^r \lambda_{n+s}^{p_s}}{\displaystyle\prod_{s=1}^r \lambda_s^{p_s}\prod_{j=1}^J \lambda_{r+j}^{C_j(p)}}
\end{multline}
and for $i=r+J+k_0$ with $1\leq k_0\leq K$ we have
\begin{multline}
G_{r+J+k_0}(\lambda) = -(\lambda_1\cdots\lambda_{r+J})^{-1} \sum_{\substack{p_1,\dots,p_r=0\\ D_{k_0}(p)\neq 0}}^\infty (-1)^{\sum_{s=1}^r p_s + \sum_{j=1}^J C_j(p)}  \\ \cdot H_{D_{k_0}(p)}
\frac{\displaystyle\prod_{j=1}^J C_j(p)!}{\displaystyle\prod_{k=1}^K D_k(p)!} \frac{\displaystyle\prod_{k=1}^K \lambda_{r+J+k}^{D_k(p)}\prod_{s=1}^r \lambda_{n+s}^{p_s}}{\displaystyle\prod_{s=1}^r \lambda_s^{p_s}\prod_{j=1}^J \lambda_{r+j}^{C_j(p)}}
\end{multline}
The specialization of Section 3 that led from $F_{v^{(0)}}(\lambda)$ to the series (1.6) turns the series (4.1) and (4.2) into (1.9) and (1.10), respectively.

In general, $v^{(0)}$ does not have minimal $\hat{\imath}$-negative support for $i=n+1,\dots,n+r$.  It does, however, if we impose the condition of Theorem~1.12(b).
\begin{proposition}  If $\sum_{i=1}^{J+r}{\bf a}_i$ is the unique interior lattice point of $(J+r+1)\Delta(A)$, then $v^{(0)}$ has minimal $\hat{\imath}$-negative support for $i=n+1,\dots,n+r$.  
\end{proposition}

\begin{proof}
To fix ideas, take $i=n+1$.  Let $l\in {L}$ be as given in (3.5).  If $p_1\geq 0$, then $(\widehat{n+1})$-${\rm nsupp}(v^{(0)}+l)$ cannot be a proper subset of $(\widehat{n+1})$-${\rm nsupp}(v^{(0)})$.  So suppose $p_1<0$.  Then 
$(\widehat{n+1})$-${\rm nsupp}(v^{(0)}+l)$ will be a proper subset of $(\widehat{n+1})$-${\rm nsupp}(v^{(0)})$ if and only if $p_2,\dots,p_r\geq 0$ and $D_k(p)\geq 0$ for $k=1,\dots,K$.  

Since $l\in {L}$ we have the equation
\[ \sum_{s=1}^r p_s{\bf a}_s + \sum_{j=1}^J C_j(p){\bf a}_{r+j} = \sum_{k=1}^K D_k(p){\bf a}_{r+J+k}+ \sum_{s=1}^r p_s{\bf a}_{n+s}, \]
which we rewrite in the form
\begin{equation}
-p_1{\bf a}_{n+1} +\sum_{s=2}^r p_s{\bf a}_s + \sum_{j=1}^J C_j(p){\bf a}_{r+j} = -p_1{\bf a}_1 + \sum_{k=1}^K D_k(p){\bf a}_{r+J+k}+ \sum_{s=2}^r p_s{\bf a}_{n+s}
\end{equation}
so that all coefficients on the right-hand side are nonnegative and $-p_1>0$.  

By the uniqueness of $\sum_{i=1}^{J+r}{\bf a}_i$, the lattice point ${\bf a}_{n+1} + \sum_{i=2}^{J+r}{\bf a}_i$ 
must lie on a codimension-one face $\sigma$ of $C(A)$.  This implies that each of the lattice points ${\bf a}_{n+1},{\bf a}_2,\dots,{\bf a}_{J+r}$ lies on $\sigma$.  It follows that the left-hand side of (4.4) lies on the hyperplane containing $\sigma$.  The right-hand side of (4.4), where all coefficients are nonnegative, then lies on $\sigma$, and, since the coefficient $-p_1$ of ${\bf a}_1$ is $>0$, it follows that ${\bf a}_1$ lies on $\sigma$.  But then $\sum_{i=1}^{J+r}{\bf a}_i$ lies on $\sigma$, contradicting the fact that it is an interior point of $(J+r+1)\Delta(A)$.
\end{proof}

\end{document}